\documentclass[11pt,brazil,english,12pt]{elsarticle}
\usepackage{libertine-type1}
\usepackage{berasans}
\usepackage{eulervm}
\usepackage[T1]{fontenc}
\usepackage[utf8]{inputenc}
\usepackage{geometry}
\usepackage{array}
\usepackage{longtable}
\usepackage{booktabs}
\usepackage{multirow}
\usepackage{amsmath}
\usepackage{amsthm}
\usepackage{amssymb}

\makeatletter

\newcommand{\lyxmathsym}[1]{\ifmmode\begingroup\def\b@ld{bold}
  \text{\ifx\math@version\b@ld\bfseries\fi#1}\endgroup\else#1\fi}

\providecommand{\tabularnewline}{\\}

\numberwithin{equation}{section}
\numberwithin{figure}{section}
\theoremstyle{plain}
\newtheorem{thm}{\protect\theoremname}
  \theoremstyle{definition}
  \newtheorem{example}[thm]{\protect\examplename}
  \theoremstyle{plain}
  \newtheorem{cor}[thm]{\protect\corollaryname}

\usepackage{graphicx}
\usepackage{setspace}
\usepackage{titlesec}

\newcommand{\mathsym}[1]{{}}
\newcommand{\unicode}[1]{{}}

\makeatletter
\def\ps@pprintTitle{%
 \let\@oddhead\@empty
 \let\@evenhead\@empty
 \def\@oddfoot{}%
 \let\@evenfoot\@oddfoot}
\makeatother
\titleformat{\section}
  {\normalfont\sffamily\Large\bfseries}
  {\thesection}{1em}{}

\titleformat{\title}
  {\normalfont\sffamily\huge\bfseries}
  {\thesection}{1em}{}

\makeatother

\usepackage{babel}
  \addto\captionsbrazil{\renewcommand{\corollaryname}{Corolário}}
  \addto\captionsbrazil{\renewcommand{\examplename}{Exemplo}}
  \addto\captionsbrazil{\renewcommand{\theoremname}{Teorema}}
  \addto\captionsenglish{\renewcommand{\corollaryname}{Corollary}}
  \addto\captionsenglish{\renewcommand{\examplename}{Example}}
  \addto\captionsenglish{\renewcommand{\theoremname}{Theorem}}
  \providecommand{\corollaryname}{Corollary}
  \providecommand{\examplename}{Example}
\providecommand{\theoremname}{Theorem}

\begin{document}

\begin{frontmatter}{}

\title{{\Huge{}Classification of three-dimensional Jordan superalgebras}}

%

\author[fn1,rvt]{M. E. Martin\corref{cor1}}

\ead{eugenia@ime.usp.br}

\fntext[fn1]{The author was supported by CAPES, PNPD Program, IME-USP.}



\address[rvt]{Instituto de Matemática e Estat\'{i}stica. Universidade de São Paulo.
R. Matão 1010, 05508-090. São Paulo, SP, Brazil. }

\global\long\def\ka{\mathbb{\mathbf{k}}}

\global\long\def\F{\mathcal{F}}

\global\long\def\P{\mathcal{P}}

\global\long\def\J{\mathcal{J}}

\global\long\def\Jor{\operatorname{\J\hspace{-0.063cm}or}}

\global\long\def\JorR{\operatorname{\J\hspace{-0.063cm}or}^{\mathbb{R}}}

\global\long\def\JorC{\operatorname{\J\hspace{-0.063cm}or}^{\mathbb{C}}}

\global\long\def\JorN{\operatorname{\J\hspace{-0.063cm}orN}}

\global\long\def\M{\mathcal{M}}

\global\long\def\U{\mathcal{U}}

\global\long\def\B{\mathcal{B}}

\global\long\def\T{\mathcal{T}}

\global\long\def\S{\mathcal{S}}

\global\long\def\G{\operatorname{G}}

\global\long\def\V{\operatorname{V}}

\global\long\def\W{\operatorname{W}}

\global\long\def\Id{\operatorname{Id}}

\global\long\def\car{\operatorname{char}}

\global\long\def\Mat{\operatorname*{Mat}}

\global\long\def\Aut{\operatorname*{Aut}}

\global\long\def\Rad{\operatorname*{Rad}}

\global\long\def\Ann{\operatorname*{Ann}}

\global\long\def\Anc{\operatorname*{Anc}}

\global\long\def\Hom{\operatorname*{Hom}}

\global\long\def\Mod{\operatorname*{Mod}}

\global\long\def\Lie{\operatorname*{Lie}}

\global\long\def\Assoc{\operatorname*{Assoc}}

\global\long\def\dsum{{\displaystyle \sum}}

\global\long\def\Gr{\operatorname{Gr}(m,n)}
 
\begin{abstract}
In this paper, we classify Jordan superalgebras of dimension up to
three over an algebraically closed field of characteristic different
of two. We also prove that all such superalgebras are special.
\end{abstract}

\end{frontmatter}{}

\section*{Introduction}

The goal of this paper is to classify Jordan superalgebras of dimension
up to three over an algebraically closed field of characteristic different
of $2$. Our main motivation to obtain such classification comes
out from the intention to give an answer to the problem of determining
the minimal dimension of exceptional Jordan superalgebras posed in\foreignlanguage{brazil}{
\cite{ExceptionalShestakov}}. This problem is analogous to the problem
presented in \cite{Dniester_aberto} by H. Petersson and A. M. Slin'ko
for ordinary Jordan algebras. Our strategy to provide a lower bounded
for this dimension is to determine the complete list of Jordan superalgebras
of small dimensions and verify which ones are special or exceptional.

The classification problem of the algebraic structures of given dimension,
i.e. find the list of all non-isomorphic objects in this structure,
has been extensively studied. But in the context of graded algebras
the results are relatively scarce. Most of the known results focus
on the theory of graded Lie algebras. In \foreignlanguage{brazil}{\cite{KacLie1977},}
V. G. Kac gave a brief account of the main features of the theory
of finite dimensional Lie superalgebras. The author announced in \foreignlanguage{brazil}{\cite{Kac1975}}
the complete solution of the classification problem of simple Lie
superalgebras under the stimuli of the physicists’ interest in the
subject.

In the context of nonsimple Lie superalgebras, in 1978 N. Backhouse
gave a classification of real Lie superalgebras, which are not Lie
algebras, up to dimension four (see \cite{BackhouseSuperLie4} and
\cite{BackhouseSuperLie3}). He also noted that none of them are simple.
In 1999, A. Hegazi in \cite{HegaziSuperLieNilpo5} proved a version
of the Kirillov lemma for nilpotent Lie superalgebras and applied
its to the classification of nilpotent Lie superalgebras up to dimension
five. Classifications of five dimensional Lie superalgebras over the
complex and real numbers are contributions of N. L. Matiadou and A.
Fellouris, see \cite[2005]{MatiadouLie5C} and \cite[2007]{MatiadouLie5R}
respectively.

As  expected, there was a similar theory for graded Jordan algebras.
The paper \cite{Kaplansky1975} of I. Kaplansky, lays the foundations
for the study of the Jordan superalgebras case. In 1977, V. G. Kac
developed the Kantor's method and applied it to obtain the classification
of finite dimensional simple Jordan superalgebras over an algebraically
closed field of characteristic $0$ from the classification of simple
graded Lie algebras (see \cite{Kac1977}). There was one missing case
in this classification that was later described by I. L. Kantor in
\cite{kantor1990}. More than two decades later in \cite{Racine2003},
M. L. Racine and E. I. Zel'manov took advantage of Jordan theoretic
methods, like Peirce decomposition and representation theory, that
are less sensitive to the characteristic of the base field to obtain
a classification of simple Jordan superalgebras over fields of characteristic
different from $2$ whose even part is semisimple. The remaining
case, namely simple Jordan superalgebras with nonsemisimple even part,
was tackled in \cite{Martinez-Zelmanov-Prime} by C. Martinez and
E. Zel'manov. 

However, nonsimple Jordan superalgebras are very plentiful and as
yet do not have received enough attention. In what follows we present
the classification of Jordan superalgebras of small dimensions and
we determine in each case whether the superalgebras obtained are special
or exceptional.

Regarding special or exceptional Jordan superalgebras we highlight
some central results. In \cite{Martinez2001} the authors construct
universal associative enveloping algebras for a large class of simple
Jordan superalgebras showing they are special. On the other hand,
in \cite{ZelmanovCounterexamples} has been proved that the $10$-dimensional
simple Kac superalgebra is exceptional and showed some new exceptional
simple superalgebras appear in characteristic $3$.

As for exceptional Jordan superalgebras that are not simple, I. Shestakov
and others in \cite[2000]{ExceptionalShestakov} have constructed an
example of an exceptional Jordan superalgebra of dimension $7$ over
an algebraically closed field of characteristic different of $2$.
So far this is the smallest known example.

The paper is organized as follows. In Section \ref{sec:Preliminaries},
we recall the basic concepts, examples and some results for finite-dimensional
Jordan superalgebras. In Section \ref{sec:Classification}, we describe
all non-isomorphic Jordan superalgebras of dimension less or equal
to $3$ and, finally, we show that all superalgebras are special.

\selectlanguage{brazil}%
\selectlanguage{english}%

\section{Preliminaries\label{sec:Preliminaries}}

In what follows $\ka$ will represent a field of characteristic not
$2$. A \emph{Jordan algebra} $\J$ is a commutative $\ka$-algebra
that satisfies the Jordan identity
\begin{equation}
a^{2}\cdot(b\cdot a)=(a^{2}\cdot b)\cdot a\label{eq:JordanIdentity}
\end{equation}
 or, equivalently, the linearization

\[
(a\text{·}b)\text{·}(c\text{·}d)+(a\text{·}c)\text{·}(b\text{·}d)+(a\text{·}d)\text{·}(b\text{·}c)=((a\text{·}b)\text{·}c)\text{·}d+((a\text{·}d)\text{·}c)\text{·}b+((b\text{·}d)\text{·}c)\text{·}a
\]
for any $a,\mbox{ }b,\mbox{ }c,\mbox{ }d\in\J$.

In \cite{martin} all Jordan algebras up to dimension four over an
algebraically closed field were described. To fix notation we present
the indecomposable Jordan algebras of dimension less or equal to $3$
in Table \ref{tab:IndecompJA3}. This algebras will be used in the
development we present below.

\footnotesize%
\begin{longtable}{c>{\centering}p{7cm}>{\centering}p{1cm}}
\caption{\label{tab:IndecompJA3}Indecomposable Jordan algebras of dimension
$\leq3$ }
\tabularnewline
\toprule 
$\J$ & Multiplication table & $\dim$\tabularnewline
\endfirsthead
\midrule 
$\U_{1}$  & $e_{1}^{2}=e_{1}$ & $1$\tabularnewline
\midrule 
$\U_{2}$  & $e_{1}^{2}=0$ & $1$\tabularnewline
\midrule 
$\mathcal{B}_{1}$  & $e_{1}^{2}=e_{1}\quad e_{1}\cdot e_{2}=e_{2}\quad e_{2}^{2}=0$ & $2$\tabularnewline
\midrule 
$\mathcal{B}_{2}$  & $e_{1}^{2}=e_{1}\quad e_{1}\cdot e_{2}=\frac{1}{2}e_{2}\quad e_{2}^{2}=0$ & $2$\tabularnewline
\midrule 
$\mathcal{B}_{3}$  & $e_{1}^{2}=e_{2}\quad e_{1}\cdot e_{2}=0\quad e_{2}^{2}=0$ & $2$\tabularnewline
\midrule 
$\mathcal{T}_{1}$  & $e_{1}^{2}=e_{1}\quad e_{2}^{2}=e_{3}\quad e_{3}^{2}=0$

$e_{1}\cdot e_{2}=e_{2}\quad e_{1}\cdot e_{3}=e_{3}\quad e_{2}\cdot e_{3}=0$  & $3$\tabularnewline
\midrule 
$\mathcal{T}_{2}$  & $e_{1}^{2}=e_{1}\quad e_{2}^{2}=0\quad e_{3}^{2}=0$

$e_{1}\cdot e_{2}=e_{2}\quad e_{1}\cdot e_{3}=e_{3}\quad e_{2}\cdot e_{3}=0$  & $3$\tabularnewline
\midrule 
$\mathcal{T}_{3}$  & $e_{1}^{2}=e_{2}\quad e_{2}^{2}=0\quad e_{3}^{2}=0$

$e_{1}\cdot e_{2}=e_{3}\quad e_{1}\cdot e_{3}=0\quad e_{2}\cdot e_{3}=0$  & $3$\tabularnewline
\midrule 
$\mathcal{T}_{4}$  & $e_{1}^{2}=e_{2}\quad e_{2}^{2}=0\quad e_{3}^{2}=0$

$e_{1}\cdot e_{2}=0\quad e_{1}\cdot e_{3}=e_{2}\quad e_{2}\cdot e_{3}=0$  & $3$\tabularnewline
\midrule 
$\mathcal{T}_{5}$  & $e_{1}^{2}=e_{1}\quad e_{2}^{2}=e_{2}\quad e_{3}^{2}=e_{1}+e_{2}$

$e_{1}\cdot e_{2}=0\quad e_{1}\cdot e_{3}=\frac{1}{2}e_{3}\quad e_{2}\cdot e_{3}=\frac{1}{2}e_{3}$ & $3$\tabularnewline
\midrule 
$\mathcal{T}_{6}$  & $e_{1}^{2}=e_{1}\quad e_{2}^{2}=0\quad e_{3}^{2}=0$

$e_{1}\cdot e_{2}=\frac{1}{2}e_{2}\quad e_{1}\cdot e_{3}=e_{3}\quad e_{2}\cdot e_{3}=0$  & $3$\tabularnewline
\midrule 
$\mathcal{T}_{7}$  & $e_{1}^{2}=e_{1}\quad e_{2}^{2}=0\quad e_{3}^{2}=0$

$e_{1}\cdot e_{2}=\frac{1}{2}e_{2}\quad e_{1}\cdot e_{3}=\frac{1}{2}e_{3}\quad e_{2}\cdot e_{3}=0$ & $3$\tabularnewline
\midrule 
$\mathcal{T}_{8}$  & $e_{1}^{2}=e_{1}\quad e_{2}^{2}=e_{3}\quad e_{3}^{2}=0$

$e_{1}\cdot e_{2}=\frac{1}{2}e_{2}\quad e_{1}\cdot e_{3}=0\quad e_{2}\cdot e_{3}=0$  & $3$\tabularnewline
\midrule 
$\mathcal{T}_{9}$  & $e_{1}^{2}=e_{1}\quad e_{2}^{2}=e_{3}\quad e_{3}^{2}=0$

$e_{1}\cdot e_{2}=\frac{1}{2}e_{2}\quad e_{1}\cdot e_{3}=e_{3}\quad e_{2}\cdot e_{3}=0$  & $3$\tabularnewline
\midrule 
$\mathcal{T}_{10}$  & $e_{1}^{2}=e_{1}\quad e_{2}^{2}=e_{2}\quad e_{3}^{2}=0$

$e_{1}\cdot e_{2}=0\quad e_{1}\cdot e_{3}=\frac{1}{2}e_{3}\quad e_{2}\cdot e_{3}=\frac{1}{2}e_{3}$  & $3$\tabularnewline
\midrule
\end{longtable}\normalsize

A \emph{superalgebra $A$} is a $\mathbb{Z}_{2}$-graded algebra,
that is, $A=A_{\overline{0}}+A_{\overline{1}}$ where $A_{\overline{0}}$
and $A_{\overline{1}}$ are subspaces of $A$ verifying $A_{\overline{i}}\cdot A_{\overline{j}}\subseteq A_{\overline{i+j}}$
for any $\overline{i},\overline{j}\in\mathbb{Z}_{2}$ (here $+$ is
a direct sum of vector spaces). Therefore, $A_{\overline{0}}$ (called
the \emph{even part} of $A$) is a subalgebra of $A$ and $A_{\overline{1}}$
(the \emph{odd part}) is a bimodule over $A_{\overline{0}}$. If $a\in A_{\overline{i}}$,
then we say that the element $a$ is \emph{homogeneous} and we will
denote by $\left|a\right|$ the \emph{parity} of $a$, that is, $\left|a\right|=i$. 

A subspace $V$ of a superalgebra $A$ is called \emph{sub-superspace}
if $V=(V\cap A_{\overline{0}})+(V\cap A_{\overline{1}})$. By a \emph{sub-superalgebra}
(or \emph{ideal}) of a superalgebra $A$ we mean any sub-superspace
that is a subalgebra (or ideal) of $A$ as an algebra. Given two superalgebras
$A$ and $B$ and a homomorphism of algebras $\phi:A\to B$, we say
$\phi$ is a \emph{homomorphism of superagebras} if it preserves the
grading, i.e., $\phi(A_{\overline{i}})\subseteq B{}_{\overline{i}}$
for $i=0,1$. A superalgebra $A$ is called \emph{simple} if $A^{2}\neq0$
and it does not contain any non-trivial (graded) ideals.

A superalgebra $J=J_{\overline{0}}+J_{\overline{1}}$ is a \emph{Jordan
superalgebra} if it satisfies the supercommutativity
\begin{equation}
a\cdot b=(-1)^{\left|a\right|\left|b\right|}b\cdot a\label{eq:superComutativity}
\end{equation}
 and the super Jordan identity
\begin{multline}
(a\cdot b)\cdot(c\cdot d)+(-1)^{\left|b\right|\left|c\right|}(a\cdot c)\cdot(b\cdot d)+(-1)^{\left|b\right|\left|d\right|+\left|c\right|\left|d\right|}(a\cdot d)\cdot(b\cdot c)=\\
((a\cdot b)\cdot c)\cdot d+(-1)^{\left|c\right|\left|d\right|+\left|b\right|\left|c\right|}((a\cdot d)\cdot c)\cdot b+(-1)^{\left|a\right|\left|b\right|+\left|a\right|\left|c\right|+\left|a\right|\left|d\right|+\left|c\right|\left|d\right|}((b\cdot d)\cdot c)\cdot a\label{eq:superJordanid}
\end{multline}
for any $a,\mbox{ }b,\mbox{ }c,\mbox{ }d\in J$.

The even part $J_{\overline{0}}$ is a Jordan algebra and the odd
part $J_{\overline{1}}$ is an $J_{\overline{0}}$-bimodule with an
operation $J_{\overline{1}}\times J_{\overline{1}}\to J_{\overline{1}}$.

Note that a superidentity is obtained from the corresponding identity
following the Kaplansky rule: any alteration in the order of odd elements
causes a change in sign. 

An associative superalgebra is just an associative algebra with a
$\mathbb{Z}_{2}$-graduation but a Jordan superalgebra is not necessarily
a Jordan algebra. 
\begin{example}
\label{exa:Kaplansky} Consider the three-dimensional simple Jordan
superalgebra due to I. Kaplansky: $K_{3}=\ka e+(\ka x+\ka y)$ with
multiplication given by $e^{2}=e$, $x^{2}=y^{2}=0$, $e\cdot x=x\cdot e=\frac{1}{2}x$,
$e\cdot y=y\cdot e=\frac{1}{2}y$ and $x\cdot y=-y\cdot x=e$. Take
$a=e+x$ and $b=y$ in identity \eqref{eq:JordanIdentity} then $(a^{2}\cdot b)\cdot a-a^{2}(b\cdot a)=a\neq0$.
 
\end{example}

Indeed, a Jordan superalgebra is a Jordan algebra with a $\mathbb{Z}_{2}$-graduation
if and only if $J_{\overline{1}}^{2}=0$. On the other hand, the same
Jordan algebra can have two or more non-isomorphic $\mathbb{Z}_{2}$-graduations. 
\begin{example}
The non-associative three-dimensional Jordan algebra $\T_{6}$ (see
Table \ref{tab:IndecompJA3}) admits four $\mathbb{Z}_{2}$-graduations
that, as we will see in Theorem \ref{thm:superalg-nao-isomorfas},
determine four non-isomorphic Jordan superalgebras: $\S_{4}^{3}$,
$\S_{9}^{3}$, $\S_{12}^{3}$ and the trivial one $\T_{6}^{s}$. 
\end{example}

In fact, an algebra $A$ admits a non-trivial $\mathbb{Z}_{2}$-graduation
if and only if there exists an automorphism $\varphi\in\Aut(A)$ such
as $\varphi^{2}=Id$ and $\varphi\neq Id$. The relation between $\varphi$
and the $\mathbb{Z}_{2}$-graduation is given by $A_{\overline{0}}=\left\{ a\in A\mid\varphi(a)=a\right\} $
and $A_{\overline{1}}=\left\{ a\in A\mid\varphi(a)=-a\right\} $.
\begin{example}
 Consider the three dimensional associative Jordan algebra $\T_{1}$.
We observe that $\Aut(\T_{1})=\left\{ \varphi:\T_{1}\to\T_{1}\mid\varphi(e_{1})=e_{1};\,\varphi(e_{2})=\alpha e_{2}+\beta e_{3};\,\varphi(e_{3})=\alpha^{2}e_{3};\,\text{with }\alpha,\beta\in\ka\right\} $
then $(\T_{1})_{\overline{0}}=\ka e_{1}+\ka e_{3}$ and $(\T_{1})_{\overline{1}}=\ka e_{2}$
is the only (up to graded isomorphism) non-trivial $\mathbb{Z}_{2}$-graduation
on $\T_{1}$. Then it is an associative superalgebra which is not
a Jordan superalgebra since it does not satisfy the supercommutativity.
\end{example}

If $A=A_{\bar{0}}+A_{\bar{1}}$ is an associative superalgebra, we
can get a Jordan superalgebra structure with the same underlying vector
space, $A^{(+)}$, by defining a new product 
\[
a\odot b=\frac{1}{2}(ab+(-1)^{\left|a\right|\left|b\right|}ba)\quad\text{ for any }a,b\in A.
\]

A Jordan superalgebra $J=J_{\overline{0}}+J_{\overline{1}}$ is said
to be \emph{special} is $J$ is a sub-superalgebra of $A^{(+)}$,
for some associative superalgebra $A$. Otherwise, it is called \emph{exceptional}.
We call $A$ an \emph{associative envelope} for $J$. The \emph{universal
associative envelope} $U(J)$ of the Jordan superalgebra $J$ is the
unital associative superalgebra satisfying the following universal
property, which implies that $U(J)$ is unique up to isomorphism:
there is a morphism of Jordan superalgebras $\alpha:J\to U(J)^{(+)}$
such that for any unital associative superalgebra $A$ and any morphism
of Jordan superalgebras $\beta:J\to A^{(+)}$, there is a unique morphism
of associative superalgebras $\gamma:U(J)\to A$ satisfying $\beta=\gamma\lyxmathsym{◦}\alpha$.

\begin{example}
\label{exa:K3-is-special}The \emph{Kaplansky superalgebra} $K_{3}$
of Example \ref{exa:Kaplansky} is special, in fact it is a sub-su\-per\-al\-ge\-bra
of $M_{1\mid1}^{(+)}(W_{1})$ where $W_{1}=\operatorname{alg}\left\langle \xi,\eta:\xi\cdot\eta-\eta\cdot\xi=1\right\rangle $
is the Weyl algebra. The even part and the odd part are respectively,
\[
\left\{ \left(\begin{array}{cc}
a & 0\\
0 & b
\end{array}\right)\mid a,b\in W_{1}\right\} \text{ and }\left\{ \left(\begin{array}{cc}
0 & c\\
d & 0
\end{array}\right)\mid c,d\in W_{1}\right\} ,
\]
 then the application of $K_{3}$ in $M_{1\mid1}^{(+)}(W_{1})$ such
as $e\mapsto e_{11}$, $x\mapsto2e_{12}+2\xi e_{21}$ and $y\mapsto\eta e_{12}+\xi\eta e_{21}$
is an homomorphism.
\end{example}

\begin{example}
\label{exa:superform-is-special}(\emph{Superalgebra of a superform})
Let $V=V_{\overline{0}}\oplus V_{\overline{1}}$ be a vector superspace
over a field $\ka$, and let $f:V\times V\to\ka$ be a \emph{supersymmetric
bilinear form}, i.e., $f\mid_{V_{\overline{0}}}$ is symmetric, $f\mid_{V_{\overline{1}}}$
is skew-symmetric and $f(V_{\overline{i}},V_{\overline{j}})=0$ if
$i\neq j$. Then $J(V,f)=\ka\cdot1\oplus V$ is a Jordan superalgebra
with $J_{\overline{0}}=\ka\cdot1\oplus V_{\overline{0}}$, $J_{\overline{1}}=V_{\overline{1}}$
and the product given by $(\alpha1+v)\cdot(\beta1+w)=(\alpha\beta+f(v,w))\cdot1+(\alpha w+\beta v)$.
 Moreover, suppose that $f$ is non-degenerate, $\dim V_{\overline{0}}=n$
and $\dim V_{\overline{1}}=2m$. It follows from \cite{Martinez2001}
that $J(V,f)$ is an special Jordan superalgebra. In fact $J(V,f)\subseteq(\operatorname{Cl}(n)\widetilde{\otimes}_{\ka}W_{m})^{(+)}$,
where $\operatorname{Cl}(n)=\left\langle 1,e_{1},\dots e_{n}\mid e_{i}e_{j}+e_{j}e_{i}=0,\,i\neq j,\,e_{i}^{2}=1\right\rangle $
is the Clifford superalgebra and $W_{m}=\left\langle 1,x_{i},y_{i},1\leq i\leq m\mid[x_{i},y_{j}]=\delta_{ij},[x_{i},x_{j}]=[y_{i},y_{j}]=0\right\rangle $
is the Weyl superalgebra with $\left|e_{i}\right|=\left|x_{j}\right|=\left|y_{k}\right|=1$.

\end{example}

\begin{example}
Algebra $\T_{6}$ in Table \ref{tab:IndecompJA3} is special as an
ordinary Jordan algebra. Its universal associative envelope is the
five-dimensional algebra given by 
\[
\begin{array}{c|ccccc}
U(\T_{6}) & 1 & e_{1} & e_{2} & e_{3} & e_{4}\\
\hline 1 & 1 & e_{1} & e_{2} & e_{3} & e_{4}\\
e_{1} & e_{1} & e_{1} & e_{4} & e_{3} & e_{4}\\
e_{2} & e_{2} & e_{2}-e_{4} & 0 & 0 & 0\\
e_{3} & e_{3} & e_{3} & 0 & 0 & 0\\
e_{4} & e_{4} & 0 & 0 & 0 & 0
\end{array}
\]
with homomorphism $e_{i}\mapsto e_{i}$ for $i=1,2,3$. Now, consider
the $\mathbb{Z}_{2}$-graduation in $\T_{6}$ which result in the
superalgebra $\S_{12}^{3}$ (see Table \ref{tab:(2,1)B2}) with $\dim J_{\overline{0}}=2$
and $\dim J_{\overline{1}}=1$. Then the graduation induced on $U(\T_{6})=A_{\overline{0}}+A_{\overline{1}}$
where $A_{\overline{0}}=\{1,e_{1},e_{2},e_{4}\}$ and $A_{\overline{1}}=\{e_{3}\}$
is an associative envelope for $\S_{12}^{3}$ but is not the universal
one. In this case, $U(\S_{12}^{3})$ is the superalgebra with basis
$\left\{ 1,e_{1},e_{2},e_{1}e_{2},o_{1}^{k},\:k\geq1\right\} $ quotient
with the ideal: 
\[
\left\langle e_{2}^{2},\,e_{2}o_{1},\,o_{1}e_{2},\,e_{1}^{2}-e_{1},\,e_{1}o_{1}-o_{1},\,o_{1}e_{1}-o_{1},\,e_{2}e_{1}-e_{2}+e_{1}e_{2}\right\rangle .
\]
\end{example}

\begin{example}
In \cite{Kac1977} the $10$-dimensional \emph{Kac superalgebra} $K$
was defining. It is given by $K_{\overline{0}}=\left\langle a_{1},\dots,a_{6}\right\rangle $
and $K_{\overline{1}}=\left\langle \xi_{1},\dots,\xi_{4}\right\rangle $
with
\begin{gather*}
a_{1}\cdot a_{i}=a_{i},\quad{\scriptstyle i=1,\dots,5}\quad a_{1}\cdot a_{6}=0,\quad a_{1}\cdot\xi_{i}=\frac{1}{2}\xi_{i},\quad{\scriptstyle i=1,\dots,4};\\
a_{2}\cdot a_{3}=a_{1},\quad a_{2}\cdot\xi_{3}=\xi_{1},\quad a_{2}\cdot\xi_{4}=\xi_{2},\quad a_{3}\cdot\xi_{1}=\frac{1}{2}\xi_{3},\quad a_{3}\cdot\xi_{2}=\frac{1}{2}\xi_{4};\\
a_{4}\cdot a_{5}=a_{1},\quad a_{4}\cdot\xi_{2}=\xi_{1},\quad a_{4}\cdot\xi_{4}=\xi_{3},\quad a_{5}\cdot\xi_{1}=\frac{1}{2}\xi_{2},\quad a_{5}\cdot\xi_{3}=\frac{1}{2}\xi_{4};\\
a_{6}^{2}=a_{6},\quad a_{6}\cdot\xi_{i}=\frac{1}{2}\xi_{i},\quad{\scriptstyle i=1,\dots,4};\\
\xi_{1}\cdot\xi_{2}=a_{2},\quad\xi_{1}\cdot\xi_{3}=a_{4},\quad\xi_{1}\cdot\xi_{4}=a_{1}+a_{6};\\
\xi_{2}\cdot\xi_{3}=a_{1}+a_{6},\quad\xi_{2}\cdot\xi_{4}=a_{5},\quad\xi_{3}\cdot\xi_{4}=a_{3};
\end{gather*}

and products which are not obtained from these transpositions of factors
are equal to $0$. In \cite{ZelmanovCounterexamples} the authors
proved that $K$ is exceptional. If $\car\ka=3$ this superalgebra
is not simple, but it has a simple subalgebra of dimension $9$, called
the \emph{degenerated Kac superalgebra}. There are two other examples
of exceptional simple Jordan superalgebras in $\car\ka=3$, see \cite{Racine2003}.
\end{example}

\begin{example}
\label{exa:exceptional7}An example of non-simple exceptional Jordan
superalgebra was obtained in \cite{ExceptionalShestakov}. It is the
$7$-dimensional superalgebra with even part $J_{\overline{0}}=\T_{7}$,
odd part generated by $\{o_{1},o_{2},o_{3},o_{4}\}$ and multiplication
given by
\begin{alignat*}{1}
e_{1}\cdot o_{i} & =o_{i}\cdot e_{1}=\frac{1}{2}o_{i}\ \text{ for }i=2,3\quad n_{1}\cdot o_{j}=o_{j}\cdot n_{1}=o_{j+1}\ \text{ for }j=1,3\\
n_{2}\cdot o_{1} & =o_{1}\cdot n_{2}=o_{3}\quad n_{2}\cdot o_{2}=o_{2}\cdot n_{2}=-o_{4}\quad o_{1}\cdot o_{2}=-o_{2}\cdot o_{1}=n_{2}
\end{alignat*}
with the other product being zero.
\end{example}

\subsection{Peirce decomposition for Jordan superalgebras}

Suppose that a Jordan superalgebra $J$ possess an idempotent element
then it has an homogeneous idempotent $e$ such as $e\in J_{\overline{0}}$.
Hence \eqref{eq:superJordanid} gives
\begin{gather*}
0=((a\cdot e)\cdot e)\cdot e+((a\cdot e)\cdot e)\cdot e+a\cdot e-(a\cdot e)\cdot e-(a\cdot e)\cdot e-(a\cdot e)\cdot e\\
0=2((a\cdot e)\cdot e)\cdot e-3(a\cdot e)\cdot e+a\cdot e\\
0=[2R_{e}^{3}-3R_{e}^{2}+R_{e}](a)
\end{gather*}
for any element $a\in J$, then $2R_{e}^{3}-3R_{e}^{2}+R_{e}=0$ (where
$R_{e}:J\to J$ denote the right multiplication by $e$). Using the
standard argument for the Peirce decomposition for Jordan algebras
(see \cite{jacobson}), we obtain:
\begin{thm}
Let $e\in J$ be an homogeneous idempotent element, then $J$ is a
direct sum of vector spaces 
\begin{equation}
J=\mathcal{\mathcal{P}}_{0}\oplus\mathcal{\mathcal{P}}_{\frac{1}{2}}\oplus\mathcal{P}_{1},\text{ where }\ensuremath{\mathcal{P}_{i}=\left\{ x\in J\mid x\cdot e=e\cdot x=ix\right\} }\text{ with }\ensuremath{i=0,\,\frac{1}{2},\,1}\label{eq:Peirce1}
\end{equation}
which have the following multiplicative properties:\textup{
\begin{align}
\mathcal{\mathcal{P}}_{1}^{2} & \subseteq\mathcal{P}_{1},\quad\mathcal{\mathcal{P}}_{1}\cdot\mathcal{\mathcal{P}}_{0}=(0),\quad\mathcal{\mathcal{P}}_{0}^{2}\subseteq\mathcal{\mathcal{P}}_{0}\text{,}\label{eq:Peirce1-1}\\
\mathcal{\mathcal{P}}_{0}\cdot\mathcal{\mathcal{P}}_{\frac{1}{2}} & \subseteq\mathcal{\mathcal{P}}_{\frac{1}{2}},\quad\mathcal{\mathcal{P}}_{1}\cdot\mathcal{\mathcal{P}}_{\frac{1}{2}}\subseteq\mathcal{\mathcal{P}}_{\frac{1}{2}},\quad\mathcal{\mathcal{P}}_{\frac{1}{2}}^{2}\subseteq\mathcal{\mathcal{P}}_{0}+\mathcal{\mathcal{P}}_{1}\text{.}\label{eq:Peirce1-2}
\end{align}
}
\end{thm}

Indeed, put $b=d=e$ in \eqref{eq:superJordanid} and take $a\in\P_{i}$
and $c\in\P_{j}$ to obtain
\begin{eqnarray}
(1-2i)(a\cdot c)\cdot e & = & (-1)^{\left|a\right|\left|c\right|}j(c\cdot a)-2ij(a\cdot c).\label{eq:Peirce1-3}\\
 & = & (-1)^{2\left|a\right|\left|c\right|}j(a\cdot c)-2ij(a\cdot c)=j(1-2i)(a\cdot c)
\end{eqnarray}
 If $i=j=0$ we get $(a\cdot c)\cdot e=0$ and if $i=j=1$ we have
$(a\cdot c)\cdot e=a\cdot c$. Also $i=0,\:j=1$ and $i=1,\:j=0$
give $(a\cdot c)\cdot e=(a\cdot c)$ and $(a\cdot c)\cdot e=0$, respectively.
Hence $0=(a\cdot c)\cdot e=(-1)^{\left|a\right|\left|c\right|}(c\cdot a)\cdot e=(-1)^{\left|a\right|\left|c\right|}(c\cdot a)$,
thus $c\cdot a=0=a\cdot c$ and \eqref{eq:Peirce1-1} holds. Similarly,
$i=0,1$ and $j=\frac{1}{2}$ in \eqref{eq:Peirce1-3} give $\mathcal{\mathcal{P}}_{i}\cdot\mathcal{\mathcal{P}}_{\frac{1}{2}}\subseteq\mathcal{\mathcal{P}}_{\frac{1}{2}}$
for $i=0,1$. Now, putting $c=d=e$ and taking $a,b\in\P_{\frac{1}{2}}$
in \eqref{eq:superJordanid} we obtain
\[
(a\cdot b)\cdot e+\frac{1}{2}(a\cdot b)=((a\cdot b)\cdot e)\cdot e+\frac{1}{4}a\cdot b+(-1)^{\left|a\right|\left|b\right|}\frac{1}{4}b\cdot a=((a\cdot b)\cdot e)\cdot e+\frac{1}{2}a\cdot b.
\]
Writing $a\cdot b=c_{0}+c_{1}+c_{\frac{1}{2}}$ where $c_{i}\in\P_{i}$
we have $\left(\frac{1}{2}-\frac{1}{4}\right)c_{\frac{1}{2}}=0$.
Hence $c_{\frac{1}{2}}=0$ and $a\cdot b\in\P_{0}+\P_{1}$. Thus \eqref{eq:Peirce1-2}
holds.

More generally, we have the refined Peirce decomposition:
\begin{thm}
If $e=\sum_{i=1}^{n}e_{i}$, is a sum of pairwise orthogonal idempotents,
then $J$ decomposes into a direct sum of subspaces
\begin{equation}
J=\bigoplus_{0\leq i\leq j\leq n}\mathcal{\mathcal{P}}_{ij}\label{eq:Peircen}
\end{equation}
where
\begin{align*}
\mathcal{P}_{00} & =\left\{ x\in\mathcal{J}\mid x\cdot e_{i}=0\text{ for all }i\right\} ,\\
\mathcal{P}_{ii} & =\left\{ x\in\mathcal{J}\mid x\cdot e_{i}=x,\,x\cdot e_{k}=0\text{ for }k\neq i\text{ and }i\neq0\right\} ,\\
\mathcal{P}_{0i} & =\left\{ x\in\mathcal{J}\mid x\cdot e_{i}=\frac{1}{2}x,\,x\cdot e_{k}=0\text{ for }k\neq i\text{ and }i\neq0\right\} ,\\
\mathcal{P}_{ij} & =\left\{ x\in\mathcal{J}\mid x\cdot e_{i}=x\cdot e_{j}=\frac{1}{2}x\:i\neq j,\,x\cdot e_{k}=0\text{ for }0\neq i<j\text{ and }k\neq i,j\right\} 
\end{align*}
\end{thm}

Note that $\P_{00}=\cap_{i=1}^{n}\P_{0}(e_{i})$, $\P_{ii}=\P_{1}(e_{i})\cap_{\underset{k\neq i}{k=1}}^{n}\P_{0}(e_{k})$,
$\P_{0i}=\P_{\frac{1}{2}}(e_{i})\cap_{\underset{k\neq i}{k=1}}^{n}\P_{0}(e_{k})$,
$\P_{ij}=\P_{\frac{1}{2}}(e_{i})\cap\P_{\frac{1}{2}}(e_{j})\cap_{\underset{k\neq i,j}{k=1}}^{n}\P_{0}(e_{k})$
for $0\neq i<j$ and $\P_{00}=\P_{0i}=0$ if $e=1$ is unity in $J$.
Then using the previous decomposition relative to one idempotent we
have that the Peirce components $\mathcal{\mathcal{P}}_{ij}$ satisfy
the following relations:
\begin{align*}
\mathcal{\mathcal{\mathcal{P}}}_{ii}^{2} & \subseteq\mathcal{\mathcal{\mathcal{P}}}_{ii},\quad\mathcal{\mathcal{\mathcal{P}}}_{ij}\cdot\mathcal{\mathcal{\mathcal{P}}}_{ii}\subseteq\mathcal{\mathcal{\mathcal{P}}}_{ij},\quad\mathcal{\mathcal{\mathcal{P}}}_{ij}^{2}\subseteq\mathcal{\mathcal{\mathcal{P}}}_{ii}+\mathcal{\mathcal{\mathcal{P}}}_{jj},\\
\mathcal{\mathcal{\mathcal{P}}}_{ij}\cdot\mathcal{\mathcal{\mathcal{P}}}_{jk} & \subseteq\mathcal{\mathcal{\mathcal{P}}}_{ik},\quad\mathcal{\mathcal{\mathcal{P}}}_{ii}\cdot\mathcal{\mathcal{\mathcal{P}}}_{jj}=\mathcal{\mathcal{\mathcal{P}}}_{ii}\cdot\mathcal{\mathcal{\mathcal{P}}}_{jk}=\mathcal{\mathcal{\mathcal{P}}}_{ij}\cdot\mathcal{\mathcal{\mathcal{P}}}_{kl}=(0)
\end{align*}
where the indices $i$, $j$, $k$, $l$ are all distinct.

Peirce decompositions are inherited by ideals or by subalgebras containing
$e_{1},\dots,e_{n}$. The latter is the case of $J_{\overline{0}}$
since the idempotents are supposed to be homogeneous. On the other
hand, note that this also happens in $J_{\overline{1}}$. Indeed,
taking $x\in J_{\overline{1}}$ we can write $x=\varSigma p_{ij}$
with $p_{ij}\in\P_{ij}$ , $0\leq i\leq j\leq n$. Multiplying successively
$x$ by appropriate idempotents $e_{i}$ and using the fact that $e_{i}\cdot J_{\overline{1}}\subseteq J_{\overline{1}}$
we get $p_{ij}\in J_{\overline{1}}$.  Therefore $J_{\overline{1}}=\bigoplus_{0\leq i\leq j\leq n}\P_{ij}^{\overline{1}}$
where $\P_{ij}^{\overline{1}}=J_{\overline{1}}\cap\P_{ij}$.

\section{Classification of Jordan superalgebras\label{sec:Classification}}

In the next sections we will present the complete classification of
Jordan superalgebras of dimensions $1,2$ and $3$. From now on $\ka$
will be understood as an algebraically closed field of $\car\neq2$.
Henceforth, for convenience we drop $\cdot$ and denote the multiplication
in a superalgebra $J$ simply as $ab$. Products that we do not specify
are understood to be $0$ or be determined by commutativity or anticommutativity. 

We will say that $J=J_{\overline{0}}+J_{\overline{1}}$ is a Jordan
superalgebra of type $(n,m)$ if $n=\dim J_{\overline{0}}$ and $m=\dim J_{\overline{1}}$.
We will denote by $e_{i}$ the elements in $J_{\overline{0}}$ and
by $o_{i}$ the ones which belong to $J_{\overline{1}}$. It follows
from the supercommutativity \eqref{eq:superComutativity} that $o_{i}^{2}=0$
for $i=1,\dots,m$. Every indecomposable superalgebra obtained in
the classification will be denoted by $\S_{j}^{i}$ where the superscript
$i$ represent its dimension. 

The description of Jordan superalgebras will be organized according
to all possible values of $n$ and $m$. For example, we will look
at the most trivial cases. Suppose that $n=0$, then $\{o_{1},\dots,o_{m}\}$
is a basis for $J=J_{\overline{1}}$ an $m$-dimensional Jordan superalgebra
of type $(0,m)$. Since $o_{i}o_{j}\in J_{\overline{0}}=0$ for any
$1\leq i,j\leq m$ then there exists only one such $J$ and it is
a direct sum of $m$ copies of $\S_{1}^{1}=\ka o_{1}$ where $o_{1}^{2}=0$.
On the other hand, the Jordan superalgebras of type $(n,0)$ correspond
to the ordinary $n$-dimensional Jordan algebras with the trivial
$\mathbb{Z}_{2}$-graduation. We denote trivial indecomposable superalgebras
by $\U_{i}^{s}$, $\B_{j}^{s}$ and $\T_{k}^{s}$ with multiplication
given, respectively, by $\U_{i}$, $\B_{j}$ and $\T_{k}$ in Table
\ref{tab:IndecompJA3} for $i=1,2$, $j=1,2,3$ and $1\leq k\leq10$.
Henceforward, we will suppose $n$ and $m$ both not null.

\subsection{Jordan superalgebras of dimension $2$}

Suppose that $J=J_{\overline{0}}+J_{\overline{1}}$ is a Jordan superalgebra
of type $(1,1)$ and let $\{e_{1},o_{1}\}$ be a basis of $J$. According
to Table \ref{tab:IndecompJA3}, we have two possibilities for the
even part $J_{\overline{0}}=\U_{1}$ or $J_{\overline{0}}=\U_{2}$.
Since $o_{1}^{2}=0$ we only have to determinate $\beta\in\ka$ such
as $e_{1}o_{1}=\beta o_{1}$.

Writing the Jordan superidentity \eqref{eq:superJordanid} for the
basis in the first case, i.e. $J_{\overline{0}}=\U_{1}$, we obtain
only one equation $(\beta-1)\beta(2\beta-1)o_{1}=0$ with solutions
$\beta=0,\frac{1}{2},1$ which defines, respectively, three non-isomorphic
superalgebras: $\U_{1}^{s}\oplus\S_{1}^{1}$, $\S_{1}^{2}$ and $\S_{2}^{2}$
(see Table \ref{tab:JAvsJSA-2}). Proceeding in the same way in case
of $J_{\overline{0}}=\U_{2}$ we still obtain one equation $2\beta^{3}o_{1}=0$,
then necessarily $\beta=0$ and the superalgebra is $\U_{2}^{s}\oplus\S_{1}^{1}$.

Every Jordan superalgebra of dimension two satisfies $J_{\overline{1}}^{2}=0$
so they are ordinary Jordan algebras. Since non-isomorphic Jordan
algebras have non-isomorphic $\mathbb{Z}_{2}$-graduations, superalgebras
in different lines in Table \ref{tab:JAvsJSA-2} are non-isomorphic.
On the other hand, isomorphism of superalgebras preserves the grading
then superalgebras in different columns in Table \ref{tab:JAvsJSA-2}
are also non-isomorphic. Consequently, all two-dimensional Jordan
superalgebras here presented are pairwise non-isomorphics.\footnotesize%
\begin{longtable}{c>{\centering}p{2.5cm}>{\centering}p{4.5cm}>{\centering}p{2.5cm}}
\caption{\label{tab:JAvsJSA-2}Two-dimensional Jordan superalgebras}
\tabularnewline
\endfirsthead
\toprule 
\multirow{2}{*}{Algebra} & \multicolumn{3}{c}{Superalgebras}\tabularnewline
\cmidrule{2-4} 
 & Type $(0,2)$ & Type $(1,1)$ & Type $(2,0)$\tabularnewline
\midrule
$\mathcal{B}_{1}$  &  & $\S_{2}^{2}:\;e_{1}^{2}=e_{1}\ e_{1}o_{1}=o_{1}$ & $\B_{1}^{s}$\tabularnewline
$\mathcal{B}_{2}$  &  & $\S_{1}^{2}:\;e_{1}^{2}=e_{1}\ e_{1}o_{1}=\frac{1}{2}o_{1}$ & $\B_{2}^{s}$\tabularnewline
$\mathcal{B}_{3}$  &  &  & $\B_{3}^{s}$\tabularnewline
$\U_{1}\oplus\U_{1}$ &  &  & $\U_{1}^{s}\oplus\U_{1}^{s}$\tabularnewline
$\U_{1}\oplus\U_{2}$ &  & $\U_{1}^{s}\oplus\S_{1}^{1}$ & $\U_{1}^{s}\oplus\U_{2}^{s}$\tabularnewline
$\U_{2}\oplus\U_{2}$ & $\S_{1}^{1}\oplus\S_{1}^{1}$ & $\U_{2}^{s}\oplus\S_{1}^{1}$ & $\U_{2}^{s}\oplus\U_{2}^{s}$\tabularnewline
\midrule
\end{longtable}\normalsize

\subsection{Jordan superalgebras of dimension $3$\label{subsec:classification3}}

\subsubsection{Jordan superalgebras of type $(1,2)$}

Let $J$ be a three-dimensional Jordan superalgebra of type $(1,2)$.
The even part, $J_{\overline{0}}$, is one of the one-dimensional
Jordan algebras $J_{\overline{0}}=\U_{1}$ or $J_{\overline{0}}=\U_{2}$.
If $\{\bar{e}_{1},\bar{o}_{1},\bar{o}_{2}\}$ is a basis of $J$,
we must determine the products
\[
\bar{e}_{1}\bar{o}_{1}=\alpha\bar{o}_{1}+\beta\bar{o}_{2},\quad\bar{e}_{1}\bar{o}_{2}=\gamma\bar{o}_{1}+\delta\bar{o}_{2}\quad\bar{o}_{1}\bar{o}_{2}=\varepsilon\bar{e}_{1}.
\]
\begin{enumerate}
\item Suppose that $J_{\overline{0}}=\U_{2}$. We claim that $J$ is one
of the superalgebras in Table \ref{tab:(1,2)J3}. Indeed, the Jordan
superidentity \eqref{eq:superJordanid} for the basis provides the
system of equations 
\begin{align*}
2\left(\alpha^{3}+2\alpha\beta\gamma+\beta\gamma\delta\right)\bar{o}_{1}+2\beta\left(\alpha^{2}+\alpha\delta+\beta\gamma+\delta^{2}\right)\bar{o}_{2} & =0\\
2\varepsilon(\alpha\delta-\beta\gamma)\bar{e}_{1} & =0\\
\varepsilon\left((\alpha-\delta)^{2}+4\beta\gamma\right)\bar{e}_{1} & =0\\
2\gamma\left(\alpha^{2}+\alpha\delta+\beta\gamma+\delta^{2}\right)\bar{o}_{1}+2\left(\alpha\beta\gamma+2\beta\gamma\delta+\delta^{3}\right)\bar{o}_{2} & =0\\
\varepsilon\left(\alpha^{2}-\alpha\delta+2\beta\gamma\right)\bar{o}_{1}+\beta\varepsilon(\alpha+\delta)\bar{o}_{2} & =0\\
\varepsilon(\delta(\delta-\alpha)+2\beta\gamma)\bar{o}_{2}+\gamma\varepsilon(\alpha+\delta)\bar{o}_{1} & =0
\end{align*}
with the following solutions:

\begin{enumerate}
\item $\gamma=-\frac{\alpha^{2}}{\beta}\text{ and }\delta=-\alpha$. Suppose
$\varepsilon\neq0$ then $J\simeq\S_{1}^{3}$: if $\alpha\neq0$ the
isomorphism is given by $\bar{e}_{1}\mapsto e_{1}$, $\bar{o}_{1}\mapsto-\frac{\sqrt{\beta}}{\alpha\sqrt{\varepsilon}}o_{2}$
and $\bar{o}_{2}\mapsto\frac{\alpha}{\sqrt{\beta\varepsilon}}o_{1}+\sqrt{\frac{\beta}{\varepsilon}}o_{2}$,
but if $\alpha=0$ then $\bar{e}_{1}\mapsto e_{1}$, $\bar{o}_{1}\mapsto-\frac{1}{\sqrt{\beta\varepsilon}}o_{1}$
and $\bar{o}_{2}\mapsto-\sqrt{\frac{\text{\ensuremath{\beta}}}{\varepsilon}}o_{2}$.
In case of $\varepsilon=0$, we obtain $J\simeq\S_{3}^{3}$ with change
of basis $\bar{e}_{1}\mapsto e_{1}$, $\bar{o}_{1}\mapsto-\frac{\beta}{\alpha^{2}}o_{2}$
and $\bar{o}_{2}\mapsto o_{1}+\frac{\beta}{\alpha}o_{2}$ for $\alpha\neq0$
or $\bar{e}_{1}\mapsto e_{1}$, $\bar{o}_{1}\mapsto o_{1}$ and $\bar{o}_{2}\mapsto\beta o_{2}$
if $\alpha=0$. 

\item $\alpha=\beta=\gamma=\delta=0$. If we suppose $\varepsilon\neq0$,
the superalgebra $J$ is $\S_{2}^{3}$ with isomorphism $\bar{e}_{1}\mapsto\varepsilon e_{1}$,
$\bar{o}_{1}\mapsto o_{1}$ and $\bar{o}_{2}\mapsto o_{2}$. Otherwise,
$J\simeq\U_{2}^{s}\oplus\S_{1}^{1}\oplus\S_{1}^{1}$.
\item $\alpha=\beta=\delta=0$. If $\gamma=0$ we are in the previous case,
so we can consider $\gamma\neq0$. Suppose further that $\varepsilon\neq0$
then we obtain the superalgebra $\S_{1}^{3}$ with isomorphism given
by $\bar{e}_{1}\mapsto e_{1}$, $\bar{o}_{1}\mapsto-\frac{i}{\sqrt{\gamma\varepsilon}}o_{2}$
and $\bar{o}_{2}\mapsto-i\sqrt{\frac{\gamma}{\varepsilon}}o_{1}$.
But if $\varepsilon=0$, then again we get $J\simeq\S_{3}^{3}$ with
change of basis $\bar{e}_{1}\mapsto e_{1}$, $\bar{o}_{1}\mapsto o_{2}$
and $\bar{o}_{2}\mapsto\gamma o_{1}$.
\end{enumerate}
\footnotesize%
\begin{longtable}{ccc}
\caption{\label{tab:(1,2)J3}Jordan superalgebras of type $(1,2)$ with $J_{\overline{0}}=\protect\U_{2}$.}
\tabularnewline
\endfirsthead
\toprule 
$J$ & Multiplication table & Observation\tabularnewline
\midrule
$\S_{1}^{3}$ & $e_{1}o_{1}=o_{2}\:o_{1}o_{2}=e_{1}$ & \tabularnewline
$\S_{2}^{3}$ & $o_{1}o_{2}=e_{1}$ & associative\tabularnewline
$\S_{3}^{3}$ & $e_{1}o_{1}=o_{2}$ & associative\tabularnewline
 & $\U_{2}^{s}\oplus\S_{1}^{1}\oplus\S_{1}^{1}$  & associative\tabularnewline
\midrule
\end{longtable}\normalsize
\item Suppose that $J_{\overline{0}}=\U_{1}$, then $J$ has an even homogeneous
idempotent element. Using the Peirce decomposition \eqref{eq:Peirce1}
we get $J=\P_{0}\oplus\P_{1}\oplus\P_{\frac{1}{2}}$ and the corresponding
decomposition of the odd part $J_{\overline{1}}=\P_{0}^{\overline{1}}\oplus\P_{1}^{\overline{1}}\oplus\P_{\frac{1}{2}}^{\overline{1}}$
where $\P_{i}^{\overline{1}}=J_{\overline{1}}\cap\P_{i}$, $i=0,\frac{1}{2},1$.
Then $J$ is completely defined by the product $o_{1}o_{2}=\varepsilon e_{1}$
and the Peirce subspace $\P_{i}^{\overline{1}}$ to which $o_{1}$
and $o_{2}$ belong. In this way we obtain eight non-isomorphic superalgebras
given in Table \ref{tab:(1,2)J2}.

\footnotesize%
\begin{longtable}{ccc}
\caption{\label{tab:(1,2)J2}Jordan superalgebras of type $(1,2)$ with $J_{\overline{0}}=\protect\U_{1}$.}
\tabularnewline
\endfirsthead
\toprule 
$J$ & Multiplication table & Observation\tabularnewline
\midrule
 & $\S_{1}^{2}\oplus\S_{1}^{1}$ & $o_{1}\in\P_{0},o_{2}\in\P_{\frac{1}{2}}$ \tabularnewline
$\S_{4}^{3}$ & $e_{1}^{2}=e_{1}\:e_{1}o_{1}=o_{1}\:e_{1}o_{2}=\frac{1}{2}o_{2}$ & $o_{1}\in\P_{1},o_{2}\in\P_{\frac{1}{2}}$ \tabularnewline
 & $\S_{2}^{2}\oplus\S_{1}^{1}$ & $o_{1}\in\P_{0},o_{2}\in\P_{1}$, associative\tabularnewline
 & $\U_{1}^{s}\oplus\S_{1}^{1}\oplus\S_{1}^{1}$ & $o_{1},o_{2}\in\P_{0}$ associative\tabularnewline
$\S_{5}^{3}$ & $e_{1}^{2}=e_{1}\:e_{1}o_{1}=\frac{1}{2}o_{1}\:e_{1}o_{2}=\frac{1}{2}o_{2}$ & $o_{1},o_{2}\in\P_{\frac{1}{2}}\:\varepsilon=0$\tabularnewline
$\S_{6}^{3}$ & $e_{1}^{2}=e_{1}\:e_{1}o_{1}=o_{1}\:e_{1}o_{2}=o_{2}$ & $o_{1},o_{2}\in\P_{1}\:\varepsilon=0$ associative\tabularnewline
$\S_{7}^{3}$ & $e_{1}^{2}=e_{1}\:e_{1}o_{1}=\frac{1}{2}o_{1}\:e_{1}o_{2}=\frac{1}{2}o_{2}\:o_{1}o_{2}=e_{1}$ & $o_{1},o_{2}\in\P_{\frac{1}{2}}\:\varepsilon\neq0$ \tabularnewline
$\S_{8}^{3}$ & $e_{1}^{2}=e_{1}\:e_{1}o_{1}=o_{1}\:e_{1}o_{2}=o_{2}\:o_{1}o_{2}=e_{1}$ & $o_{1},o_{2}\in\P_{1}\:\varepsilon\neq0$ \tabularnewline
\midrule
\end{longtable}\normalsize

Up to permutation of $o_{1}$ and $o_{2}$, we have six possibilities:
\begin{enumerate}
\item $o_{1},o_{2}\in\P_{\frac{1}{2}}$. If $\varepsilon=0$ we get the
superalgebra $\S_{5}^{3}$, but if $\varepsilon\neq0$ we can choose
$\varepsilon=1$ thus $J\simeq\S_{7}^{3}$. 
\item $o_{1},o_{2}\in\P_{1}$. Analogously, if $\varepsilon=0$ we obtain
$\S_{6}^{3}$, otherwise we can choose $\varepsilon=1$ and we get
the superalgebra $\S_{8}^{3}$.
\item $o_{1},o_{2}\in\P_{0}$, thus it follows from the multiplicative properties
of Peirce components \eqref{eq:Peirce1-1} that $o_{1}o_{2}\in\P_{0}^{2}\subseteq\P_{0}=J_{\overline{1}}$
but $o_{1}o_{2}\in J_{\overline{0}}$ then $o_{1}o_{2}=0$ and we
obtain the superalgebra $\U_{1}^{s}\oplus\S_{1}^{1}\oplus\S_{1}^{1}$.
\item $o_{1}\in\P_{0},o_{2}\in\P_{1}$, then $o_{1}o_{2}\in\P_{0}\P_{1}=0$
and $J\simeq\S_{2}^{2}\oplus\S_{1}^{1}$.
\item $o_{1}\in\P_{0},o_{2}\in\P_{\frac{1}{2}}$, it follows that $o_{1}o_{2}\in\P_{0}\P_{\frac{1}{2}}\subseteq\P_{\frac{1}{2}}=\ka o_{2}$
 thus $o_{1}o_{2}=0$ and we get $J\simeq\S_{1}^{2}\oplus\S_{1}^{1}$.
\item $o_{1}\in\P_{1},o_{2}\in\P_{\frac{1}{2}}$, analogously $o_{1}o_{2}\in\P_{1}\P_{\frac{1}{2}}\subseteq\P_{\frac{1}{2}}=\ka o_{2}$
 then $o_{1}o_{2}=0$ and we obtain the superalgebra $\S_{4}^{3}$.
\end{enumerate}
\end{enumerate}

\subsubsection{Jordan superalgebras of type $(2,1)$}

Let $J=J_{\overline{0}}+J_{\overline{1}}$ be a three-dimensional
Jordan superalgebra of type $(2,1)$ and let $\{e_{1},e_{2},o_{1}\}$
be a basis of $J$. In this case, the even part $J_{\overline{0}}$
is one of the six two-dimensional Jordan algebras given in Table \ref{tab:IndecompJA3}
and $J_{\overline{1}}=\S_{1}^{1}$. All that remains is to determine
the action of $J_{\overline{0}}$ in $o_{1}$, namely, to found scalars
$\alpha,\beta\in\ka$ such as $e_{1}o_{1}=\alpha o_{1}$ and $e_{2}o_{1}=\beta o_{1}$.
\begin{enumerate}
\item Suppose $J_{\overline{0}}=\B_{1}$, so we have one idempotent $e_{1}$
relative to which $J_{\overline{1}}=\P_{0}^{\overline{1}}\oplus\P_{1}^{\overline{1}}\oplus\P_{\frac{1}{2}}^{\overline{1}}$
with $e_{1},e_{2}\in\P_{1}$. We claim that $J$ is one of the superalgebras
in Table \ref{tab:(2,1)B1}. 

\footnotesize%
\begin{longtable}{ccc}
\caption{\label{tab:(2,1)B1}Jordan superalgebras of type $(2,1)$ with $J_{\overline{0}}=\protect\B_{1}$.}
\tabularnewline
\endfirsthead
\toprule 
$J$ & Multiplication table & Observation\tabularnewline
\midrule
 & $\B_{1}^{s}\oplus\S_{1}^{1}$ & $o_{1}\in\P_{0}$, associative\tabularnewline
$\S_{9}^{3}$ & $e_{1}^{2}=e_{1}\:e_{1}e_{2}=e_{2}\:e_{1}o_{1}=\frac{1}{2}o_{1}$ & $o_{1}\in\P_{\frac{1}{2}}$ \tabularnewline
$\S_{10}^{3}$ & $e_{1}^{2}=e_{1}\:e_{1}e_{2}=e_{2}\:e_{1}o_{1}=o_{1}$ & $o_{1}\in\P_{1}$, associative\tabularnewline
\midrule
\end{longtable}\normalsize

Indeed, according to the Peirce subspace $\P_{i}$ to which $o_{1}$
belongs we have:
\begin{enumerate}
\item $o_{1}\in\P_{0}$ then $e_{2}o_{1}\in\P_{1}\P_{0}=0$ and we obtain
the superalgebra $\B_{1}^{s}\oplus\S_{1}^{1}$.
\item $o_{1}\in\P_{\frac{1}{2}}$, writing the Jordan superidentity \eqref{eq:superJordanid}
for the basis we get the equations $-\beta^{2}o_{1}=0$ and $2\beta^{3}o_{1}=0$,
then necessarily $\beta=0$ and $J\simeq\S_{9}^{3}$. 
\item $o_{1}\in\P_{1}$, proceeding in the same way the equation $2\beta^{3}o_{1}=0$
gives $\beta=0$, thus $J\simeq\S_{10}^{3}$.
\end{enumerate}
\item If $J_{\overline{0}}=\B_{2}$, analogously we have one idempotent
$e_{1}$ relative to which $J_{\overline{1}}=\P_{0}^{\overline{1}}\oplus\P_{1}^{\overline{1}}\oplus\P_{\frac{1}{2}}^{\overline{1}}$
with $e_{1}\in\P_{1}$, $e_{2}\in\P_{\frac{1}{2}}$. Whatever the
Peirce subspace that $o_{1}$ belongs to, it follows from the multiplicative
properties \eqref{eq:Peirce1-1} and \eqref{eq:Peirce1-2} and the
fact that $e_{2}o_{1}\in J_{\overline{1}}$ that the product $e_{2}o_{1}$
is always zero. So, we obtain three non-isomorphic Jordan superalgebras
given in Table \ref{tab:(2,1)B2}.

\footnotesize%
\begin{longtable}{ccc}
\caption{\label{tab:(2,1)B2}Jordan superalgebras of type $(2,1)$ with $J_{\overline{0}}=\protect\B_{2}$.}
\tabularnewline
\endfirsthead
\toprule 
$J$ & Multiplication table & Observation\tabularnewline
\midrule
 & $\B_{2}^{s}\oplus\S_{1}^{1}$ & $o_{1}\in\P_{0}$\tabularnewline
$\S_{11}^{3}$ & $e_{1}^{2}=e_{1}\:e_{1}e_{2}=\frac{1}{2}e_{2}\:e_{1}o_{1}=\frac{1}{2}o_{1}$ & $o_{1}\in\P_{\frac{1}{2}}$\tabularnewline
$\S_{12}^{3}$ & $e_{1}^{2}=e_{1}\:e_{1}e_{2}=\frac{1}{2}e_{2}\:e_{1}o_{1}=o_{1}$ & $o_{1}\in\P_{1}$\tabularnewline
\midrule
\end{longtable}\normalsize

\item Consider $J_{\overline{0}}=\B_{3}$, since $\B_{3}$ is a nilpotent
algebra $J$ does not have any idempotent element, therefore Peirce
decomposition is not available. Then, computing the superidentity
in the basis elements we get the system
\begin{flalign*}
\alpha\left(2\alpha^{2}-3\beta\right)o_{1} & =0\quad\beta\left(2\alpha^{2}-\beta\right)o_{1}=0\\
2\alpha\beta^{2}o_{1} & =0\quad2\beta^{3}o_{1}=0
\end{flalign*}
whose solution is $\alpha=\beta=0$ (still when $\car\ka=3$). Consequently,
we obtain the associative (then Jordan) superalgebra $\B_{3}^{s}\oplus\S_{1}^{1}$.

\item Now, if $J_{\overline{0}}=\U_{1}\oplus\U_{1}$ then $J$ has two even
homogeneous orthogonal idempotent elements which implies in the following
decomposition of $J_{\overline{1}}$ in Peirce subspaces
\[
J_{\overline{1}}=\P_{00}^{\overline{1}}\oplus\P_{01}^{\overline{1}}\oplus\P_{02}^{\overline{1}}\oplus\P_{11}^{\overline{1}}\oplus\P_{12}^{\overline{1}}\oplus\P_{22}^{\overline{1}}\text{.}
\]
Thus $J$ is completely defined by the Peirce component $\P_{ij}^{\overline{1}}$
to which $o_{1}$ belongs. This provides us the four non-isomorphic
superalgebras: $\U_{1}^{s}\oplus\U_{1}^{s}\oplus\S_{1}^{1}$ if $o_{1}\in\P_{00}$,
$\S_{1}^{2}\oplus\U_{1}^{s}$ if $o_{1}\in\P_{01}$ (or, analogously,
$o_{1}\in\P_{02}$), $\S_{2}^{2}\oplus\U_{1}^{s}$ if $o_{1}\in\P_{11}$
(or, analogously, $o_{1}\in\P_{22}$) and the indecomposable one $\S_{13}^{3}$
given by $e_{1}^{2}=e_{1}$, $e_{2}^{2}=e_{2}$, $e_{1}o_{1}=\frac{1}{2}o_{1}$,
$e_{2}o_{1}=\frac{1}{2}o_{1}$ in the case of $o_{1}\in\P_{12}$.

\item Assume $J_{\overline{0}}=\U_{1}\oplus\U_{2}$, then $J$ has an idempotent
$e_{1}$ such as $J_{\overline{1}}=\P_{0}^{\overline{1}}\oplus\P_{1}^{\overline{1}}\oplus\P_{\frac{1}{2}}^{\overline{1}}$
with $e_{1}\in\P_{1}$ and $e_{2}\in\P_{0}$. If $o_{1}\in\P_{0}$
or $o_{1}\in\P_{\frac{1}{2}}$ necessarily $\beta$ has to be zero
for the respective superalgebras to satisfy the Jordan superidentity:
$2\beta^{3}o_{1}=0$ in the first case and $\beta^{2}o_{1}=0$, $2\beta^{3}o_{1}=0$
in the second one. We obtain the superalgebras $\U_{1}^{s}\oplus\U_{2}^{s}\oplus\S_{1}^{1}$
 and $\S_{1}^{2}\oplus\U_{2}^{s}$, respectively. Lastly, if $o_{1}\in\P_{1}$
then $e_{2}o_{1}\in\P_{0}\P_{1}=0$ and the resulting superalgebra
is $\S_{2}^{2}\oplus\U_{2}^{s}$. 

\item Finally, if $J_{\overline{0}}$ is the null algebra $\U_{2}\oplus\U_{2}$
then again Peirce decomposition is not available. Computing the Jordan
superidentity in the basis we have:
\begin{flalign*}
2\alpha^{3}o_{1} & =0\quad2\alpha^{2}\beta o_{1}=0\\
2\alpha\beta^{2}o_{1} & =0\quad2\beta^{3}o_{1}=0
\end{flalign*}
which has as solution $\alpha=\beta=0$. Hence, we obtain the associative
superalgebra $\U_{2}^{s}\oplus\U_{2}^{s}\oplus\S_{1}^{1}$.

\end{enumerate}

Until now, we have proved that any three-dimensional Jordan superalgebra
$J$ over an algebraically closed field of characteristic not $2$
is isomorphic to one of the superalgebras presented in this subsection.
The following theorem completes the algebraic classification.
\begin{thm}
\label{thm:superalg-nao-isomorfas}All superalgebras in Section \ref{subsec:classification3}
are pairwise non-isomorphic.
\end{thm}

\begin{proof}
Firstly, consider the superalgebras such as $J_{\overline{1}}^{2}=0$,
namely those algebras in Table \ref{tab:JAvsJSA-3}. Following an
analogous reasoning to that done in dimension two and taking into
account that decomposable superalgebras are guaranteed to be non-isomorphics,
we only have to verify whether $\S_{12}^{3}\overset{?}{\simeq}\S_{9}^{3}$.
The result follows from $(\S_{12}^{3})_{\overline{0}}=\B_{2}\not\simeq\B_{1}=(\S_{9}^{3})_{\overline{0}}$.

It remains to analyze superalgebras with $J_{\overline{1}}^{2}\neq0$,
namely $\S_{1}^{3}$, $\S_{2}^{3}$, $\S_{7}^{3}$ and $\S_{8}^{3}$.
They all have type $(1,2)$, the first two have even part $\U_{2}$
while $J_{\overline{0}}=\U_{1}$ in the remaining two superalgebras.
$\S_{2}^{3}$ is an associative superalgebra while $\S_{1}^{3}$ is
not. On the other hand, the Jordan superalgebra $\S_{8}^{3}$ has
a unity $e_{1}$ whereas $\S_{7}^{3}$ has not, indeed $\S_{7}^{3}$
is $K_{3}$ the Kaplansky superalgebra, one of the few examples of
finite dimensional simple nonunital Jordan superalgebras.

\footnotesize%
\begin{longtable}{c>{\centering}p{2cm}>{\centering}p{2cm}>{\centering}p{2cm}>{\centering}p{2cm}}
\caption{\label{tab:JAvsJSA-3}Three-dimensional Jordan superalgebras with
$J_{\overline{1}}^{2}=0$}
\tabularnewline
\toprule 
\multirow{2}{*}{Algebra} & \multicolumn{4}{c}{Superalgebras}\tabularnewline
\cmidrule{2-5} 
\endfirsthead
 & $(0,3)$ & $(1,2)$ & $(2,1)$ & $(3,0)$\tabularnewline
\midrule
$\mathcal{T}_{1}$  &  &  &  & $\T_{1}^{s}$\tabularnewline
\midrule
$\mathcal{T}_{2}$  &  & $\S_{6}^{3}$ & $\S_{10}^{3}$ & $\T_{2}^{s}$\tabularnewline
\midrule
$\mathcal{T}_{3}$  &  &  &  & $\T_{3}^{s}$\tabularnewline
\midrule
$\mathcal{T}_{4}$  &  & $\S_{3}^{3}$ &  & $\T_{4}^{s}$\tabularnewline
\midrule
$\mathcal{T}_{5}$  &  &  &  & $\T_{5}^{s}$\tabularnewline
\midrule
$\mathcal{T}_{6}$  &  & $\S_{4}^{3}$ & $\S_{12}^{3}$

$\S_{9}^{3}$ & $\T_{6}^{s}$\tabularnewline
\midrule
$\mathcal{T}_{7}$  &  & $\S_{5}^{3}$ & $\S_{11}^{3}$ & $\T_{7}^{s}$\tabularnewline
\midrule
$\mathcal{T}_{8}$  &  &  &  & $\T_{8}^{s}$\tabularnewline
\midrule
$\mathcal{T}_{9}$  &  &  &  & $\T_{9}^{s}$\tabularnewline
\midrule
$\mathcal{T}_{10}$  &  &  & $\S_{13}^{3}$ & $\T_{10}^{s}$\tabularnewline
\midrule
$\B_{1}\oplus\U_{1}$ &  &  & $\S_{2}^{2}\oplus\U_{1}^{s}$ & $\B_{1}^{s}\oplus\U_{1}^{s}$\tabularnewline
\midrule
$\B_{1}\oplus\U_{2}$ &  & $\S_{2}^{2}\oplus\S_{1}^{1}$ & $\S_{2}^{2}\oplus\U_{2}^{s}$

$\B_{1}^{s}\oplus\S_{1}^{1}$ & $\B_{1}^{s}\oplus\U_{2}^{s}$\tabularnewline
\midrule
$\B_{2}\oplus\U_{1}$ &  &  & $\S_{1}^{2}\oplus\U_{1}^{s}$ & $\B_{2}^{s}\oplus\U_{1}^{s}$\tabularnewline
\midrule 
 $\B_{2}\oplus\U_{2}$ &  & $\S_{1}^{2}\oplus\S_{1}^{1}$ & $\S_{1}^{2}\oplus\U_{2}^{s}$

$\B_{2}^{s}\oplus\S_{1}^{1}$ & $\B_{2}^{s}\oplus\U_{2}^{s}$\tabularnewline
\midrule
 $\B_{3}\oplus\U_{1}$ &  &  &  & $\B_{3}^{s}\oplus\U_{1}^{s}$\tabularnewline
\midrule
$\B_{3}\oplus\U_{2}$ &  &  & $\B_{3}^{s}\oplus\S_{1}^{1}$ & $\B_{3}^{s}\oplus\U_{2}^{s}$\tabularnewline
\midrule
 $\U_{1}\oplus\U_{1}\oplus\U_{1}$ &  &  &  & $\U_{1}^{s}\oplus\U_{1}^{s}\oplus\U_{1}^{s}$\tabularnewline
\midrule 
$\U_{1}\oplus\U_{1}\oplus\U_{2}$ &  &  & $\U_{1}^{s}\oplus\U_{1}^{s}\oplus\S_{1}^{1}$ & $\U_{1}^{s}\oplus\U_{1}^{s}\oplus\U_{2}^{s}$\tabularnewline
\midrule
$\U_{1}\oplus\U_{2}\oplus\U_{2}$ &  & $\U_{1}^{s}\oplus\S_{1}^{1}\oplus\S_{1}^{1}$ & $\U_{1}^{s}\oplus\U_{2}^{s}\oplus\S_{1}^{1}$ & $\U_{1}^{s}\oplus\U_{2}^{s}\oplus\U_{2}^{s}$\tabularnewline
\midrule
 $\U_{2}\oplus\U_{2}\oplus\U_{2}$ & $\S_{1}^{1}\oplus\S_{1}^{1}\oplus\S_{1}^{1}$ & $\U_{2}^{s}\oplus\S_{1}^{1}\oplus\S_{1}^{1}$ & $\U_{2}^{s}\oplus\U_{2}^{s}\oplus\S_{1}^{1}$ & $\U_{2}^{s}\oplus\U_{2}^{s}\oplus\U_{2}^{s}$\tabularnewline
\midrule
\end{longtable}\normalsize
\end{proof}
\begin{thm}
All Jordan superalgebras of dimension up to $3$ are special. 
\end{thm}

\begin{proof}
It follows from \cite{TeseRuso} that ordinary Jordan algebras up
to dimension $3$ are special, then it is clear that Jordan superalgebras
up to dimension $3$ such as $J_{\overline{1}}^{2}=0$ are also special
since the graduation on the associative enveloped algebra is induced
by the graduation of $J$. 

Associative Jordan superalgebras are trivially special, it is the
case of $\S_{2}^{3}$. Furthermore, we have already proved in Example
\ref{exa:K3-is-special} that $K_{3}\simeq\S_{7}^{3}$ is also special. 

Now, observe that $\S_{8}^{3}$ is the simple Jordan superalgebra
of the non-degenerate supersymmetric bilinear form $f:V\times V\to\ka$
where $V=V_{\overline{1}}$ is a vector superspace over $\ka$ such
as $\dim V_{\overline{1}}=2$, given by $f(o_{i},o_{i})=0$ for $i=1,2$
and $f(o_{1},o_{2})=1=-f(o_{2},o_{1})$. Therefore, the speciality
follows from Example \ref{exa:superform-is-special}. 

Finally, superalgebra $\S_{1}^{3}$ is exactly the algebra supercommutative,
solvable of index $2$ and nonnilpotent described in Example $1$
of \cite{ShestakovCounterexamples} and proved to be special by exhibiting
an embedding in the superalgebra $A^{(+)}$ for the matrix superalgebra
$A=M_{3}^{1,2}(\ka)$, where 
\[
A_{\overline{0}}=\left\{ \left(\begin{array}{cc}
a & 0\\
0 & b
\end{array}\right),a\in\ka,\ b\in M_{2}(\ka)\right\} 
\]
 and 
\[
A_{\overline{1}}=\left\{ \left(\begin{array}{cc}
0 & x\\
y & 0
\end{array}\right),\,x\in M_{1,2}(\ka),\ y\in M_{2,1}(\ka)\right\} 
\]
 putting $e_{1}=e_{23}\in A_{\overline{0}}$, $o_{1}=e_{13}+4e_{31}\in A_{\overline{1}}$,
$o_{2}=e_{1}\odot o_{1}$. 
\end{proof}
As a consequence, we obtain a bound for the minimal dimension of exceptional
Jordan superalgebras as posed in \cite{ExceptionalShestakov}. We observe
that, up to our knowledge, the dimension $7$ for the Example \ref{exa:exceptional7}
is the least dimension for the known examples of exceptional Jordan
superalgebras. 
\begin{cor}
The minimal dimension $d$ of an exceptional Jordan superalgebra satisfies
$4\leq d\leq7$.
\end{cor}

\section*{\protect\pagebreak{}References}

\bibliographystyle{mystyle}
\bibliography{library}

\end{document}